\documentclass{amsart}
\usepackage[utf8]{inputenc}
\usepackage{fullpage,amsfonts,mathpazo,amsmath,graphicx, wasysym, MnSymbol, array, tikz-cd}
\usepackage{tcolorbox}
\usepackage{blindtext}
\usepackage{multicol}
\usepackage{setspace}

\usepackage{mathtools}
\mathtoolsset{showonlyrefs}

\usepackage{hyperref}
\usepackage{enumerate}

\usepackage{biblatex}
\usepackage{graphicx,tikz}
\usepackage{pdfpages}
\usepackage{setspace}
\usepackage{tikz}

\newtheorem{theorem}{Theorem}
\newtheorem{lemma}{Lemma}

\graphicspath{ {graphics} }

\newcommand{\Z}{\mathbb{Z}}

\newcommand{\R}{\mathbb{R}}

\newcommand{\nI}{\noindent{}}

\newcommand{\lp}{\left(}
\newcommand{\rp}{\right)}

\begin{document}

\title[An Extremal Property of the Square Lattice]{An Extremal Property of the Square Lattice}


\author[P. Helms]{Paige Helms}
\address{University of Washington, Seattle, US. 98195}
\email{phelms@uw.edu}

\thanks{This work could not have been completed without the kind support of Jayadev Athreya and Stefan Steinerberger. This author would also like to thank Aisha Mechery and Albert Artiles for their thoughtful feedback. }

\begin{abstract}

\nI{Motivated} by a 2019 result of Faulhuber-Steinerberger \cite{extremal}, we demonstrate that the real square lattice $\Z^2$ exhibits the same local, extremal property as the hexagonal lattice $\Lambda$, where distances of lattice points from the `deep holes' of natural fundamental domains increase under perturbation. 
If $\Delta$ is a perturbation of the lattice $\Z^2$ with respect to the Euclidean metric, then for a fixed deep hole $p$, the summed total distance of lattice points to $p$ strictly increases, and is bounded below by a function of the distance between the lattice and its perturbation. Additionally, we show this growth is approximately preserved by convex functions.

\end{abstract}
\maketitle


\section{Introduction}
\label{introduction}
In this section we define the space of real, unimodular lattices, give our result, and introduce notation. The novelty of this paper lies in exhibiting $\Z^2$'s extremal behavior with respect to the Euclidean distance function; that $\Z^2$ is a critical point with respect to this function in the space of lattices is known in the community, and we choose to include it here for completeness. 
\nI{Euclidean} lattices are ubiquitous in many fields of math, for example, group theory \cite{groupslattices}, cryptography \cite{LatBase}\cite{crypto}, representation theory \cite{reptheorylat}, and the study of Lie groups and Lie algebras \cite{LatLie}. A \emph{lattice} $\Gamma$ in $\R^n$ is a discrete, additive subgroup of finite covolume. Every lattice $\Gamma$ can be expressed as a set of integer linear combinations of a basis vectors $\{v_1, v_2, \ldots, v_n\}$ in $\R^n$. Symbolically we can express $\Gamma$ as $$\Gamma =\left \{ \sum_{i=1}^n m_i v_i: m_i \in \Z\right \}.$$ The convex hull of the $v_i$ is a \textit{fundamental domain} for $\Gamma$ acting on $\R^n$. 

Using the column vectors $v_i$ we form a matrix $g$; we use this to express $\Gamma$ as $\Gamma = g\Z^n$. Restricting to the case where $g \in SL \left(n, \R \right)$ is the same as considering lattices of unit covolume. The \emph{space of unimodular lattices}, denoted $L \left(\R^n \right)$, is then given by the quotient $SL \left(n, \R \right)/SL \left(n, \Z \right)$, where the coset $gSL \left(n ,\Z \right)$ is identified with the lattice $g \Z^n$. This assignment is well-defined, since for any $h \in SL \left(n, \Z \right)$, $h\Z^n = \Z^n$. 

Every flat two-dimensional torus can be seen as the quotient of $\R^2$ by a lattice $\Gamma = g\Z^2$, where we can think of the resulting torus $\R^2/\Gamma$ as a parallelogram spanned by (any choice of) basis vectors of $\Gamma$ with with sides identified by Euclidean translations. Any basis for $\Gamma$ gives a matrix in $GL\left(2, \R\right)$ whose columns are the basis vectors. When we normalize lattices to have covolume $1$, we can then restrict our set of matrices to $SL\left(2, \R\right)$. If we consider tori up to rotation, we have a further quotient by the group $SO(2,\R)$, so our space of 2-dimensional tori (up to rotation and scaling) is $SO\left(2,\R\right) \backslash SL(2, \R) / SL(2, \Z)$. 
Figure ~\ref{fig:modular:surface} gives an illustration of $L(\R^2)$ as a fundamental domain for the action of $SL(2, \Z)$ on the upper half-plane $\mathbb H^2 = SO(2, \R)\backslash SL(2,\R)$. 

\begin{figure}\caption{An illustration of $L(\R^2)$, unimodular lattices up to rotation, via an image of the fundamental domain for $SL(2, \Z)$ acting on $\mathbb{H} \simeq SO\left(2,\R\right) \backslash SL(2, \R)$. }\label{fig:modular:surface}
\begin{tikzpicture}[scale=1]
   \draw(-5,0)--(5,0);
   \draw(180:1) arc (180:120:1);
   \draw(60:1) arc (60:0:1);
   \filldraw[fill=blue!20!white]  (1/2, 5)--(60:1) arc (60:120:1)--(-1/2, 5);
   \draw[dashed] (0,1)--(0,5) node[below] at (0, 1){\textcolor{green}{\tiny{$\Z^2$}}};
   \draw[green!80] (0,1) circle (.1 cm);

%

   \draw[xshift=1cm](180:1) arc (180:120:1);
   \draw[orange!80] (.5, .866025) circle (.1cm) node[below] at (.5, .866025){\tiny{$\Lambda$}};
   \draw[xshift=1cm](60:1) arc (60:0:1);
   \draw[xshift=1cm] (1/2, 5)--(60:1) arc (60:120:1)--(-1/2, 5);

   \draw[xshift=2cm](180:1) arc (180:120:1);
   \draw[xshift=2cm](60:1) arc (60:0:1);
   \draw[xshift=2cm] (1/2, 5)--(60:1) arc (60:120:1)--(-1/2, 5);
   
      \draw[xshift=3cm](180:1) arc (180:120:1);
   \draw[xshift=3cm](60:1) arc (60:0:1);
   \draw[xshift=3cm] (1/2, 5)--(60:1) arc (60:120:1)--(-1/2, 5);
   
      \draw[xshift=-1cm](180:1) arc (180:120:1);
   \draw[xshift=-1cm](60:1) arc (60:0:1);
   \draw[xshift=-1cm] (1/2, 5)--(60:1) arc (60:120:1)--(-1/2, 5);

   \draw[xshift=-2cm](180:1) arc (180:120:1);
   \draw[xshift=-2cm](60:1) arc (60:0:1);
   \draw[xshift=-2cm] (1/2, 5)--(60:1) arc (60:120:1)--(-1/2, 5);
   
      \draw[xshift=-3cm](180:1) arc (180:120:1);
   \draw[xshift=-3cm](60:1) arc (60:0:1);
   \draw[xshift=-3cm] (1/2, 5)--(60:1) arc (60:120:1)--(-1/2, 5);
   
   \foreach \x in { -3, -2, -1, 0, 1, 2, 3, } \path(\x,0)node[below]{\tiny $\x$};

\end{tikzpicture}
\end{figure}
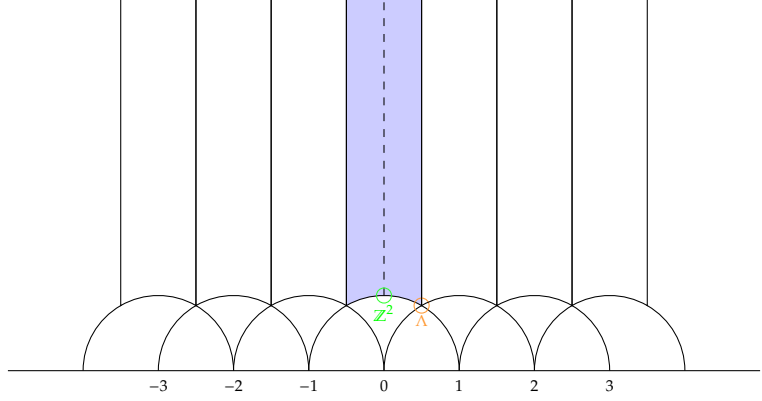

\subsection{Distances from deep holes}

We now fix our question: Let $p =\left(\frac{1}{2}, \frac{1}{2}\right)$ denote the center of the standard square fundamental domain $[0, 1]^2$ of $\Z^2$; in the terminology of \cite{extremal},  this center is called a \emph{deep hole} in the lattice. We note that deep holes are also referred to as \textit{circumcenters} when the lattices have rotational symmetry. Let $A_r\left(\Z, p\right)= \{mv+nw \in \Z^2 : |mv+nw-p| = r\}$ be the set of integer lattice points of distance $r$ from $p$. Let $\Delta = \left(v' w'\right)Z^2 = \Z v' + \Z w'$ represent a small perturbation of $\Z^2$ in the space of unimodular lattices. Then, $\det\left( v' w'\right)= 1$, and $|v-v'|$ and $|w-w'|$ are small. Next, we define $C_r\left(\Z^2, p\right)= \{mv' + nw' : |mv+nw - p| = r\}$ to be the set of perturbations of lattice points which were originally at distance $r$ from $p$ in $\Z^2$. Note that it is equivalent to express $C_r$ in terms of $A_r$: $C_r = \{mv' + nw': mv + nw \in A_r\}$. We want to compare the distances of the lattice points $C_r$ in the perturbed lattice $\Delta$ from the deep hole $p$ to the distances of the points $A_r$ in the original lattice. Symbolically, we want to compute the difference of the following sums to explore the behavior of lattices nearby $\Z^2$:
$$\sum_{\delta \in C_r} \|p-\delta\| \text{ and } \sum_{z \in A_r} \|p-z\|.$$

\subsection{Result}  

\begin{theorem}\label{theorem:theorem}
	\text{ } \\
	\indent{\indent{If}} $\Delta$ is sufficiently close to $\Z^2$ with respect to the Euclidean metric, then for a fixed deep hole $p$
	\begin{equation}\label{eq_dist2}
		\sum_{\delta \in C_r} \| p - \delta\| -  \sum_{z \in A_r} \| p - z\| \geq r \, |A_r| \, d\left(\Delta, \Z^2\right)^2.
	\end{equation} 
	\indent{\indent{If}} $\phi:\mathbb{R}_+ \rightarrow \mathbb{R}$ is any monotonically increasing, convex function, then
	\begin{equation*}\label{eq_dist3}
		\sum_{\delta \in C_r} \phi\left(\| p - \delta\|\right)-  \sum_{\lambda \in A_r} \phi\left(\| p - \lambda\|\right)\geq  r \, \phi'\left(r\right)\, |A_r| \, d\left(\Delta, \Z^2\right)^2.
	\end{equation*}
	

	\end{theorem}
\vspace{.5cm}	
The distance function $d\left(\Delta, \Z^2\right)$ is given by $\sqrt{ |v-v'|^2 + |w-w'|^2}$. Taking $d\left(\Delta, \Z^2\right)= \|\left(v, w\right)-\left(v',  w'\right)\|$, where $\| \cdot \|$ is any other norm on $\R^4$, would yield an equivalent result, up to constants. Our proof of this result relies on explicit computations of derivatives.

\subsubsection{Organization}
In this section, we gave a brief introduction to lattices and the framework of examining their distances from a fixed, non-lattice point and give our main result. In \ref{symmetries}, we prove a \ref{Lemma} about the rotational symmetry of lattice points in $\Z^2$. In \ref{section3}, we prove our result, and in \ref{sect4}, we give further directions for research which naturally arise from this result and those of \cite{extremal}.

\subsubsection{Notation}
We give a brief summary of the notation used throughout this work, some of which was given in the introduction. First, we remark that though we are always working with column vectors, for ease of notation we write them as row vectors. 
\begin{itemize}
    \item $SL \left(n, \R\right)$ is the group of $n \times n$ real matrices with determinant $1$; $SL \left(n,\Z\right)$ has integer entries.
    \item $L\left(\R^n\right)= SL\left(n, \R\right)/SL\left(n, \Z\right)$ is the space of unimodular lattices in $\R^n$. Here, we consider lattices up to rotation, so $L\left(\R^n\right)$ describes $SO\left(2, \R\right)\backslash SL\left(n, \R\right)/SL\left(n, \Z\right)$.
    \item $\Gamma$ will always refer to an arbitrary unimodular lattice. Because our lattices are integral, we can write $\Gamma = \Z a + \Z b$ in $L \left(\R^2\right)$ where $\left(a, b\right)$ is a basis for $\R^2$ with $\det\left(a, b\right)= 1$.

    \item We call the standard basis vectors $v =\left(1, 0\right)$ and $w =\left(0, 1\right)$, then we write $\Z^2 = \{ kv +  lw: k, l \in \Z\}$.
   
   \item $\Delta$ is a unimodular lattice given by a small perturbation of $\Gamma$'s basis vectors; $$\Delta = \{kv' + l w': k,l \in \Z\}$$ where $\|w -w'\|$ and $\|v-v'\|$ are small. We use $\Delta$ to denote the perturbation of both $\Gamma$ and $\Z^2$, but the context should make it clear what lattice is being perturbed. 
   
    \item $\Lambda$ is the unit covolume \emph{hexagonal} lattice. We use the basis $$\Lambda = \left\{k \frac{\sqrt{2}}{3^{1/4}} \,\left(1, \, 0\right)+ l \frac{1}{3^{1/4} \sqrt{2}} \,\left(1,\, \sqrt{3}\right):, k,l \in \Z\right\}.$$ 
    The \textit{density} of a lattice refers to the reciprocal of the covolume, meaning that $$\text{density of } \Gamma = \frac{1}{\text{vol} \left(\R^2/\Gamma\right)}$$.
    
    \item Given a lattice $\Gamma =\left(ma + n b\right)\in L \left(\R^2\right)$ and a point $q \in \R^2$, we define $$ A_r\left(\Gamma, q\right)= \{ ma+nb: m, n \in \Z, |ma+nb-q| = r\} $$ to be the set of lattice points exactly distance $r$ from $q$. We will denote it just as $A_r$ when the lattice is understood, and $A_r\left(\Gamma \right)$ to specify the lattice explicitly. 
    
    \item For $\Delta =\left(\Z a' + \Z b'\right)$ a fixed small perturbation of $\Gamma$, we define $C_r$ to be the set of perturbations of lattice points in $\Gamma$ which are distance $r$ from $p$; that is $$C_r = \{ ma'+nb': m, n \in \Z,  |ma+nb-p| = r\}.$$ The choice of fundamental domain and therefore $p$ is not reflected in this notation; the calculation is independent of this choice. 
\end{itemize}

\section{Symmetries} \label{symmetries} 
Following \cite{extremal}, we first show that points in lattice $\Z^2$ at a fixed distance $r$ from deep hole $p =\left(\frac{1}{2}, \frac{1}{2}\right)$ occur naturally in quadruples. In the following lemma, let $\Z^2 \in \R^2$ have basis 
$$ v = \left(1,0\right)\text{ and } w = \left(0,1\right).$$
The deep hole of the unit square is $p = \frac{v+w}{2} =\left(\frac{1}{2}, \frac{1}{2}\right)$ and consider $R = \begin{pmatrix} 0 & 1 \\ -1& 0 \end{pmatrix},$ the rotation matrix by $\frac{\pi}{2}$. This lemma says that for $p$ and \textit{any} lattice point $q$, there is a quadruple of lattice points given by a rotation of the vector connecting $p$ to $q$ by $\frac{\pi}{2}$; see Figure \ref{fig:vector:rotate} below.

\begin{lemma} \label{Lemma}   For any $q \in \R^2$ such that $p+q \in \Z^2$, the images $R^i q$ are in the set
$\{p+q, p + Rq, \, p + R^2q  \, p + R^3q\} \subset \Z^2$. \end{lemma}

\begin{proof} We explicitly calculate the quantities $p + Rq, \, p + R^2q, \, p + R^3q$, and argue that they are contained in $\Z^2$. If $p+q \in \Z^2$, then $ p+q = kv + lw \text{ for some } k,l\in \Z$ and therefore we can write $$q = \lp k-\frac{1}{2}\rp v + \lp l-\frac{1}{2}\rp w. $$ Using that
$Rv = w$ and $Rw = -v$, we have that
\begin{align*} 
 p + Rq &= p + R\left(\left(k-\frac{1}{2}\right)v + \left(l-\frac{1}{2}\right)w\right)\\
& = p + \left(k-\frac{1}{2}\right)Rv + \left(l-\frac{1}{2}\right)Rw\\
& = \frac{v+w}{2} + \left(k-\frac{1}{2}\right)w + \left(l-\frac{1}{2}\right)\left(-v  \right)\\
&= \frac{1}{2}v + \frac{1}{2}w + kw \\
& =\frac{1}{2}w -lw + \frac{1}{2}v\\
& = \left(1-l\right)v + kw \in \Z^2 \\
\end{align*}
Similarly, 
$ \, p + R^2 q = \left(1-k\right)v + \left(1-l\right)w \in \Z^2$, and 
$p + R^3 q = lv + \left(1-k\right)w \in \Z^2.$ 
\end{proof}

\begin{figure}\caption{Rotation of horizontal basis vector by $R$. }\label{fig:vector:rotate}
    \begin{tikzpicture}
        \draw[help lines, color=gray!30, dashed](-0.3, -0.3) grid (3.3, 3.3); 
    
    \draw[ ->] (1,1)--(2,1) node[below] at (2, 1){q};
    \draw[green!80, thick, ->] (1,1)--(1,2) node[left] at (1,2){Rq};
    \draw[thick, ->] (1.8,1.3) arc (0:89:.5);

        \node[below] at (1,1){p};

    \end{tikzpicture} 

\end{figure}
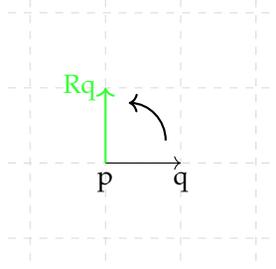

\section{Proof of Theorem} \label{section3}
To prove our main theorem using explicit computation of derivatives. To prove Theorem~\ref{theorem:theorem}, we need to show that for any lattice $\Delta$ sufficiently close $\Z^2$, 
	\begin{equation}\label{eq_dist2}
		\sum_{\delta \in C_r} \| p - \delta\| -  \sum_{z \in A_r} \| p - z\| \gtrsim r \, |A_r| \, d\left(\Delta, \Z^2\right)^2.
	\end{equation}
We note that "sufficiently close" is with respect to the Euclidean metric on the space of lattices defined in section 1. The cases of linear distance and squared distances are addressed separately and when considered together the result follows.

\subsection{Squared Distance} We now prove Theorem~\ref{theorem:theorem}. First, note that since the question we are considering is rotationally invariant. Any covolume one, unimodular lattice in $\R^2$ can be rotated to have one of its basis vectors be horizontal; we fix this as a convention. Since $\Gamma$ has covolume $1$, we can express a basis for it as
$$ v_1\left(x,y\right)= y^{\frac{-1}{2}}\left(1,0\right) \hspace{1cm}  \text{  and } \hspace{1cm} w_1\left(x,y\right)= y^{\frac{-1}{2}}\left(x,y\right).$$
To check their understanding, a reader could verify that $\Z^2$ corresponds to a parameter choice of $x=0, y=1$, and so $\Z^2$ is generated by  $v = \left(1,0\right), \, w = \left(0,1\right)$. It is standard for $\Z^2$ to consider the fundamental domain $[0, 1]^2$ with deep hole $p = \left(\frac{1}{2}, \frac{1}{2}\right)$. An arbitrary lattice point $z \in \Z^2$ is given by the expression $z = kv + lw$ for $k,l \in \Z$. It naturally has three distinct associated points by a rotation of $\frac{\pi}{2}$ around $p$. These associated points have the following expression:
\begin{align*}
&z' = p + Rq = p + \begin{bmatrix} 0 & -1 \\ 1 & 0\end{bmatrix}\left(z-p\right),\\
&z'' = p + \begin{bmatrix} 0 & -1 \\ -1 & 0\end{bmatrix}\left(z-p\right) \\
&z''' = p + \begin{bmatrix} 0 & 1 \\ -1 & 0\end{bmatrix}\left(z-p\right).\\
\end{align*}

Our naming convention follows \cite{extremal}. Now, take $z \ \in \Z^2$ and it's associated quadruple $\{z, z', z'', z'''\}$. We will investigate this quadruple under perturbation. Let $\Delta$ be a perturbation of $\Z^2$ and $\left(\delta, \delta', \delta'', \delta'''\right)$ be the perturbation in $\Delta$ of our quadruple $\left(z, z', z'', z'''\right)$ in $\Z^2$. The previous lemma implies that the tuples $\left(\delta, \delta', \delta'', \delta'''\right)$ of perturbed lattice points are of the form:
\begin{align*}
 &\delta = kv_1 + lw_1, \\
& \delta' = \left(1-l\right)v_1 + kw_1, \\
&\delta'' = \left(1-k\right)v_1 + \left(1-l\right)w_1  \\
& \delta''' = lv_1 + \left(1-k\right)w_1.\\
\end{align*}

\subsubsection{Defining $f$}
 We will show that, in total, the squared distance of a perturbed quadruple to our fixed $p$ strictly increases. That is, if $\Delta$ has parameters $x$ and $y$ as discussed above, we want to understand the behavior of the function $f(x,y)$ given by 
$$ f\left(x,y\right)= \| \delta - p\|^2 + \| \delta' - p\|^2 + \| \delta'' - p\|^2 + \| \delta''' - p\|^2.  $$

There are many ways to express and simplify $f$. We like the form
\begin{align*} 
 f\left(x,y\right)&= \left(\left(\frac{-1}{2} + \frac{k}{\sqrt{y}} + \frac{lx}{\sqrt{y}}\right)^2 + \left(\frac{1}{2} + \left(-1+k\right)\sqrt{y}\right)^2 + \left(\frac{-1}{2} + k\sqrt{y}\right)^2 \right)4y  + \left(-2 + 2k + 2\left(-1 + l\right)x + \sqrt{y}\right)^2 \\ & \hspace{.5cm}  + \left(\frac{1}{2} + \left(-1 + l\right)\sqrt{y}\right)^2 + \left(\frac{-1}{2} + l\sqrt{y}\right)^2  + \left(-2l + 2\left(-1 + k\right)x + \sqrt{y}\right)^2 + \left(-2 + 2l-2kx + \sqrt{y}\right)^2. \\
\end{align*}

\subsubsection{Partial derivatives of $f$}
We show that $\Z^2$ is a critical point in $L\left(\R^2\right)$ by directly computing partial derivatives of $f$. 
\begin{align*}
 \partial_x f &=
 \frac{\left(k-1\right)\left(-2 l+2 \left(k-1\right)x+\sqrt{y}\right)}{y} 
+\frac{\left(l-1\right)\left(-2+2 k+2\left(l-1\right)
x  +\sqrt{y}\right)}{y}\\ &-\frac{k \left(-2+2 l-2 k x+\sqrt{y}\right)}{y} +
\frac{2 l \left(-\frac{1}{2}+\frac{k}{\sqrt{y}}+\frac{l x}{\sqrt{y}}\right)}{\sqrt{y}}  \\
\end{align*}

\begin{align*}
 \partial_y f &= 
2 \left(-\frac{k}{2 y^{3/2}}-\frac{l x}{2 y^{3/2}}\right)\left(-\frac{1}{2}+\frac{k}{\sqrt{y}}+\frac{l x}{\sqrt{y}}\right)-\frac{\left(-2
l+2\left(-1+k\right)x+\sqrt{y}\right)^2}{4 y^2}  -\frac{\left(-2+2 l-2 k x+\sqrt{y}\right)^2}{4 y^2} \\
&-\frac{\left(-2+2 k+2\left(-1+l\right)x+\sqrt{y}\right)^2}{4 y^2} 
+\frac{-2
l+2\left(-1+k\right)x+\sqrt{y}}{4 y^{3/2}} + \frac{-2+2 l-2 k x+\sqrt{y}}{4 y^{3/2}}+\frac{-2+2 k+2\left(-1+l\right)x+\sqrt{y}}{4 y^{3/2}}
\\
&+ \frac{\left(-1+k\right)\left(\frac{1}{2}+\left(-1+k\right)
\sqrt{y}\right)}{\sqrt{y}} + \frac{k \left(-\frac{1}{2}+k \sqrt{y}\right)}{\sqrt{y}} + \frac{\left(-1+l\right)\left(\frac{1}{2}+\left(-1+l\right)\sqrt{y}\right)}{\sqrt{y}}+\frac{l
\left(-\frac{1}{2}+l \sqrt{y}\right)}{\sqrt{y}}. \\
\end{align*}

We evaluate both partials at $p$, or $x=0, y=1$, to show that they are identically $0$, independent of the values of $k$ and $l$. 
\begin{align*}
  \partial_x\left(0,1\right) &= 
  \left(-1 + k\right)\left(1 - 2 l\right)+\left(-1 + 2 k\right)\left(-1 + l\right)+ 2\left(-\left(\frac{1}{2}\right)+ k\right)l - k \left(-1 + 2 l\right)\\
    & = \left(-1 + k + 2l - 2kl\right)+ \left(1-2k -l +2kl\right)+ \left(k-l\right)\\
    & = 0 \\
 \end{align*}

 \begin{align*}
 \partial_y\left(0,1\right) &= 
  \left(-1+k\right)\left(-\frac{1}{2}+k\right)
+\frac{1}{4}\left(-1+2 k\right)-\frac{1}{4}\left(-1+2 k\right)^2+\frac{1}{4}\left(1-2 l\right)
 -\frac{1}{4}\left(1-2 l\right)^2 \\ 
 &+\left(-1+l\right)\left(-\frac{1}{2}+l\right)+
-\frac{l}{2}-l^2 +\frac{1}{4}\left(-1+2 l\right)-\frac{1}{4}\left(-1+2 l\right)^2\\
& = \left(\frac{1}{2}-\frac{3 k}{2}+k^2\right)+ \left(-\frac{1}{2}+\frac{3 k}{2}-k^2\right)+ \left(\frac{l}{2}-l^2\right)+ \left(\frac{1}{2}-\frac{3 l}{2}+l^2\right)+ \left(-\frac{1}{2}+l\right)\\
& = 0 \\
\end{align*}

\nI{We} have now shown $\Z^2$ is a critical point with respect to the squared distance metric $f$ on the space of lattices.

\subsubsection{The Hessian of $f$}

To understand the nature of this critical point, we compute the Hessian for $f$ at $x=0, y=1$. The Hessian has the generic form $$H\left(k,l\right)= \begin{bmatrix} \partial_{xx} & \partial_{xy} \\ \partial_{yx} & \partial_{yy}   \end{bmatrix}. $$
In our case, we have $$ H\left(k,l\right)= D^2f|_{x=0, y=1} = \begin{bmatrix} h_1\left(k,l\right)& -1 \\ -1 & h_3\left(k,l\right)\end{bmatrix}$$
where $h_1\left(k,l\right)= 4 \left(1-k+k^2-l+l^2\right)$ and $h_3\left(k,l\right)= 3-4 k+4 k^2-4 l+4 l^2$.

\nI{It} is important to note here that this is a key difference between $\Z^2$ and $\Delta$, the triangular lattice in \cite{extremal}. In $\Delta$'s Hessian, $h_3 = h_1$; here, $h_3 = h_1 - 1$. The characteristic polynomial of $H(k,l)$ is $$P\left(\lambda\right)= -1+\left(3-4 k+4 k^2-4 l+4 l^2-\lambda\right)\left(4\left(1-k+k^2-l+l^2\right)-\lambda\right)$$ with roots 
\begin{align*}
 \lambda &= \frac{1}{2} \left(7\pm \sqrt{5}-8 k+8 k^2-8 l+8 l^2\right)\\ 
& = \frac{1}{2}\left( 7 \pm \sqrt{5} + 2h_1\left(k,l\right)-8\right)\\
& = \frac{1}{2}\left(-1 \pm \sqrt{5} +  2h_1\left(k,l\right)\right).\\
\end{align*}

Let $\lambda_{min}\left(k,l\right)$ denote the behavior of smaller of the two eigenvalues, where here $h_1 + \frac{-1 - \sqrt{5}}{2}$; we want to minimize this with respect to $k$ and $l$. For all values of $\left(k,l\right)\in \Z^2$, $h_1 \geq 4$ with equality achieved at $\left(k,l\right)\in \{\left(0,0\right), \left(0,1\right), \left(1,0\right), \left(1,1\right)\}$; so, our minimum is $\lambda_{min} = 4 - \frac{-1-\sqrt{5}}{2}$. 
We note that the same conditions for $k,l$ hold for the larger root, so $\lambda_{max}$ is $h_1\left(k,l\right)+ \frac{\sqrt{5} -1}{2}$.
Then, our growth is bounded away from $0$. 
Note that for nonzero radii contained in the closure of the fundamental domain, $r = \frac{1}{\sqrt{2}}$. 
Since $\lambda_{min} = 4 - \frac{-1-\sqrt{5}}{2}$, we have that $\lambda_{min} > \frac{1}{2} = r^2 = k^2 + l^2$. \\

Lastly, we consider the asymptotic behavior of $\lambda_{min}$. We note that both $h_1$ and $h_3$ are positive definite quadratic forms, and that the off-diagonal terms are fixed at $-1$. Thus, $\lambda_{min} \geq \frac{1}{2}\left(9-\sqrt{2} + \sqrt{5}\right)$.
We now have an explicit lower bound for growth under perturbation: $\left(4 - \frac{-1-\sqrt{5}}{2}\right)- \sqrt{\frac{1}{2}}  =  \simeq 4.9109\dots$. \\
This implies the result in the case of squared distances: 
\begin{align*}
& \sum_{\delta \in C_r} \| p - \delta\|^2 -  \sum_{z \in A_r} \| p - z\|^2 \gtrsim r^2 \, |A_r| \, d\left(\Delta, \Z^2\right)^2. \\
\end{align*}
\qed

\subsection{Distance} The argument for linear distance follows the one preceding for squared distances, but here things look a little more complicated with the square root being taken over each summed term of $f\left(x,y\right)$. Beginning with $$ f\left(x,y\right)= \| \gamma - p\| + \| \gamma' - p\| + \| \gamma'' - p\| + \| \gamma''' - p\|,  $$ 
we substitute $\lambda$, $\lambda'$, $\lambda''$, $\lambda'''$ in $\Z^2$ for $\gamma$, $\gamma'$, $\gamma''$, $\gamma'''$, and the value of $p$, we re-express $f$ as the following: 

\begin{align*} 
f\left(x,y\right)&= \| \left(\frac{2k + 2lx - \sqrt{y}}{2\sqrt{y}}, \frac{2l\sqrt{y} -1}{2}\right)\| \\
&+ \|\left(\frac{2\left(1-l\right)+ 2kx - \sqrt{y}}{2 \sqrt{y}}, \frac{2k\sqrt{y} - 1}{2} \right)\|\\
& + \|\left(\frac{2\left(1-k\right)+ 2x\left(1-l\right)- \sqrt{y}}{2 \sqrt{y}}, \frac{2\left(1-l\right)\sqrt{y}}{2} \right)\| \\
& + \|\left(\frac{2l + 2\left(1-k\right)x - \sqrt{y}}{2 \sqrt{y}}, \frac{2\left(1-k\right)\sqrt{y} - 1}{2} \right)\|.  \\
\end{align*}

The next lines are the result of expanding the norm on each term. 
\begin{align*} 
f\left(x,y\right) &= 
 \left(\sqrt{\left(-\frac{1}{2}  +\frac{k}{\sqrt{y}}+\frac{l x}{\sqrt{y}}\right)^2+\left(-\frac{1}{2}+l \sqrt{y}\right)^2}\right)\\
 &+ \left(\sqrt{\left(-\frac{1}{2}+k \sqrt{y}\right)^2+\frac{\left(-2+2 l-2 k x+\sqrt{y}\right)^2}{4 y}}\right) \\
&+ \left(\sqrt{\left(\frac{1}{2}+\left(-1+l\right)\sqrt{y}\right)^2+\frac{\left(-2+2 k+2\left(-1+l\right)x+\sqrt{y}\right)^2}{4 y}}\right)\\
&+ \left(\sqrt{\left(\frac{1}{2}+\left(-1+k\right)\sqrt{y}\right)^2+\frac{\left(-2 l+2\left(-1+k\right)x+\sqrt{y}\right)^2}{4 y}}\right).  
\end{align*}

 \subsubsection{Partial Derivatives}
  
\noindent{The} partial $\partial_x f$ comes out to be 
\begin{align*}
 \partial_x f &= 
 \frac{\left(-1+k\right)\left(-2 l+2\left(-1+k\right)x+\sqrt{y}\right)}{2 \sqrt{\left(\frac{1}{2}+\left(-1+k\right)\sqrt{y}\right)^2  + \frac{\left(-2 l+2\left(-1+k\right)x+\sqrt{y}\right)^2}{4
y}} y}  - \frac{k \left(-2+2 l-2 k x+\sqrt{y}\right)}{2 \sqrt{\left(-\frac{1}{2}+k \sqrt{y}\right)^2  + \frac{\left(-2+2 l-2 k x+\sqrt{y}\right)^2}{4 y}}
y}\\ 
&+ \frac{\left(-1+l\right)\left(-2+2 k+2\left(-1+l\right)x+\sqrt{y}\right)}{2 \sqrt{\left(\frac{1}{2}+\left(-1+l\right)\sqrt{y}\right)^2+\frac{\left(-2+2 k+2\left(-1+l\right)x+\sqrt{y}\right)^2}{4
y}} y}  +\frac{l \left(-\frac{1}{2}+\frac{k}{\sqrt{y}}+\frac{l x}{\sqrt{y}}\right)}{\sqrt{\left(-\frac{1}{2}+\frac{k}{\sqrt{y}}+\frac{l x}{\sqrt{y}}\right)^2+\left(-\frac{1}{2}+l
\sqrt{y}\right)^2} \sqrt{y}}.
\end{align*}
Evaluated at $\left(0,1\right)$, we see that
\begin{align*}
    \partial_x f\left(0,1\right)&= 
     \frac{\left(-1+k\right)\left(-2 l+1\right)}{2 \sqrt{\left(\frac{1}{2}+\left(-1+k\right)\right)^2  + \frac{\left(-2 l+2\left(-1+k\right)0+1\right)^2}{4}}}  - \frac{k \left(-2+2l+\right)}{2 \sqrt{\left(-\frac{1}{2}+k \right)^2  + \frac{\left(-2+2 l+1\right)^2}{4}}}\\  
    & + \frac{\left(-1+l\right)\left(-2+2k+1\right)}{2 \sqrt{\left(\frac{1}{2}+\left(-1+l\right)\right)^2+\frac{\left(-2+2k+1\right)^2}{4}}}  +\frac{l\left(-\frac{1}{2}+k\right)}{\sqrt{\left(-\frac{1}{2}+k\right)^2+\left(-\frac{1}{2}+l\right)^2}} \\
    &  = 0. \\
\end{align*}
Next, we compute the partial of $f$ with respect to $y$, which gives us 
\begin{align*}   
 \partial_y f\left(0,1\right)& = 
     \frac{-\frac{\left(-2 l+2\left(-1+k\right)x+\sqrt{y}\right)^2}{4 y^2}+\frac{-2 l+2\left(-1+k\right)x+\sqrt{y}}{4 y^{3/2}}+\frac{\left(-1+k\right)\left(\frac{1}{2}+\left(-1+k\right)
    \sqrt{y}\right)}{\sqrt{y}}}{2 \sqrt{\left(\frac{1}{2}+\left(-1+k\right)\sqrt{y}\right)^2+\frac{\left(-2 l+2\left(-1+k\right)x+\sqrt{y}\right)^2}{4 y}}} \\
    &+ \frac{-\frac{\left(-2+2
    l-2 k x+\sqrt{y}\right)^2}{4 y^2}+\frac{-2+2 l-2 k x+\sqrt{y}}{4 y^{3/2}}+\frac{k \left(-\frac{1}{2}+k \sqrt{y}\right)}{\sqrt{y}}}{2 \sqrt{\left(-\frac{1}{2}+k
    \sqrt{y}\right)^2+\frac{\left(-2+2 l-2 k x+\sqrt{y}\right)^2}{4 y}}} + \frac{2 \left(-\frac{k}{2 y^{3/2}}-\frac{l x}{2 y^{3/2}}\right)\left(-\frac{1}{2}+\frac{k}{\sqrt{y}}+\frac{l
    x}{\sqrt{y}}\right)+\frac{l \left(-\frac{1}{2}+l \sqrt{y}\right)}{\sqrt{y}}}{2 \sqrt{\left(-\frac{1}{2}+\frac{k}{\sqrt{y}}+\frac{l x}{\sqrt{y}}\right)^2+\left(-\frac{1}{2}+l
    \sqrt{y}\right)^2}} \\
    & + \frac{-\frac{\left(-2+2 k+2\left(-1+l\right)x+\sqrt{y}\right)^2}{4 y^2}+\frac{-2+2 k+2
   \left(-1+l\right)x+\sqrt{y}}{4 y^{3/2}}+\frac{\left(-1+l\right)\left(\frac{1}{2}+\left(-1+l\right)\sqrt{y}\right)}{\sqrt{y}}}{2 \sqrt{\left(\frac{1}{2}+\left(-1+l\right)\sqrt{y}\right)^2+\frac{\left(-2+2
    k+2\left(-1+l\right)x+\sqrt{y}\right)^2}{4 y}}}. 
\end{align*}

Evaluating at $\left(0,1\right)$ we again get $0$:
\begin{align*}
 \partial_y f\left(0,1\right)&= 
 \frac{-\frac{\left(-2 l+\sqrt{y}\right)^2}{4}+\frac{-2 l+\sqrt{y}}{4}+\frac{\left(-1+k\right)\left(\frac{1}{2}+\left(-1+k\right)
\right)}{1}}{2 \sqrt{\left(\frac{1}{2}+\left(-1+k\right)\right)^2+\frac{\left(-2 l+1\right)^2}{4}}} 
+ \frac{-\frac{\left(-2+2
l+1\right)^2}{4}+\frac{-2+2 l+1}{4}+\frac{k \left(-\frac{1}{2}+k \right)}{1}}{2 \sqrt{\left(-\frac{1}{2}+k
\right)^2+\frac{\left(-2+2 l+1\right)^2}{4}}} \\ 
& + \frac{-\frac{\left(-2+2 k+1\right)^2}{4}+\frac{-2+2 k+1}{4 }+\frac{\left(-1+l\right)\left(\frac{1}{2}+\left(-1+l\right)\right)}{\sqrt{y}}}{2 \sqrt{\left(\frac{1}{2}+\left(-1+l\right)\right)^2+\frac{\left(-2+2
k+1\right)^2}{4 }}}  + \frac{2 \left(-\frac{k}{2}\right)\left(-\frac{1}{2}+k\right)+\frac{l \left(-\frac{1}{2}+l\right)}{1}}{2 \sqrt{\left(-\frac{1}{2}+\frac{k}{1}\right)^2+\left(-\frac{1}{2}+l\right)^2}} \\
 &   = 0. \hspace{4cm}
\end{align*}

\nI{Therefore} $\Z^2$ is a critical point for $f$. For our next step, we give the mixed partial $\partial g_{xy}|_{x=0, y=1}$, which evaluates to:  $$\frac{-1+k^2\left(10-24 l\right)+6 l-14 l^2+8 l^3+8 k^3\left(-1+2 l\right)-2 k \left(1-12 l^2+8 l^3\right)}{2 \sqrt{2} \left(1-2 k+2 k^2-2 l+2 l^2\right)^{3/2}}.$$

\subsubsection{The Hessian}
To establish $\Z^2$ as a local minima or maxima, we form the Hessian.  
$$ H\left(k,l\right)= D^2g|_{x=0, y=1} = \begin{bmatrix} h_1\left(k,l\right)& h_2\left(k,l\right)\\ h_2\left(k,l\right)& h_3\left(k,l\right)\end{bmatrix},$$
where 
\begin{align*}
        h_1\left(k,l\right)&= \frac{\sqrt{2}\left(1 - 3 k + 7 k^2 - 8 k^3 + 4 k^4 - 3 l + 7 l^2 - 8 l^3 + 
   4 l^4\right)}{\left(1 - 2 k + 2 k^2 - 2 l + 2 l^2\right)^{\frac{3}{2}}} \\
  h_2\left(k,l\right)&= \frac{-1 + k^2\left(10 - 24 l\right)+ 6 l - 14 l^2 + 8 l^3 + 8 k^3\left(-1 + 2 l\right)- 
 2 k\left(1 - 12 l^2 + 8 l^3\right)}{2 \sqrt{2}\left(1 - 2 k + 2 k^2 - 2 l + 
   2 l^2\right)^{\frac{3}{2}}} \\
   \text{ and \,\  } h_3\left(k,l\right)&= \frac{5 - 16 k^3 + 8 k^4 - 18 l + 26 l^2 - 16 l^3 + 8 l^4 - 
 6 k\left(3 - 8 l + 8 l^2\right)+ 
 k^2\left(26 - 48 l + 48 l^2\right)}{2 \sqrt{2}\left(1 - 2 k + 2 k^2 - 2 l + 
   2 l^2\right)^{\frac{3}{2}}}. \\
\end{align*}

\nI{The} determinant of $H\left(k,l\right)$ is 
\begin{align*}
\label{det}
& \frac{19 - 192 k^5 + 64 k^6 - 82 l + 194 l^2 - 304 l^3 + 320 l^4 - 
   192 l^5 + 64 l^6 + 64 k^4\left(5 - 3 l + 3 l^2\right)}{8\left(1 - 2 k + 
     2 k^2 - 2 l + 2 l^2\right)^2}\\
   &  - \frac{16 k^3\left(21 - 26 l + 24 l^2\right)+ 
   2 k^2\left(121 - 264 l + 336 l^2 - 192 l^3 + 96 l^4\right)- 
   2 k\left(49 - 144 l + 216 l^2 - 176 l^3 + 96 l^4\right))}{8\left(1 - 2 k + 
     2 k^2 - 2 l + 2 l^2\right)^2}. \\
     \end{align*}
     
For all values $k, \, l \in \Z$, $min_{\left(k,l\right)}det\left(H\left(k,l\right)\right) = \frac{19}{8}$ and this minimum is achieved by the three triples $\{\left(0,0\right), \left(0,1\right), \left(1,0\right)\}$. We note that $h_1>0$ for all values of $k,l$; the minimum is achieved at one of the triples minimizing determinant: $k=1, l=1$.
Since $det\left(H\left(k,l\right)\right)>0$ and $h_1>0$ for all pairs $\left(k,l\right)$, we conclude that our critical point $\Z^2 \in L\left(\R^2\right)$ is a local minimum! To establish a strictly positive lower bound on the growth of distances from lattice points to $p$ as their distance from $\Z^2$ increases, we give the following computation. \\

The characteristic polynomial $char\left(H\left(k,l\right)\right)$ is $\left(h_1 - z\right)*\left(h_3 - z\right)- h_2^2 = 0$, which expanded has a frankly hilarious form taking 10 printed lines, and so we leave them in short form:  
$$z= \frac{1}{2}\left(h_1 + h_3 \pm \sqrt{h_1^2 + 4 h_2^2 - 2h_1h_3 + h_3^2}\right).$$ 
With respect to $k$ and $l$, we claim these roots are always strictly positive. To see this, we first minimize over real values $\left( k,l\right) \in \R^2$ to identify candidates for $min_{\left(k,l\right) \in \Z^2} \left( \frac{1}{2}\left(h_1 + h_3 \pm \sqrt{h_1^2 + 4 h_2^2 - 2h_1h_3 + h_3^2}\right) \right)$. In the table \ref{table:roots} below, we give approximate values for roots, again emphasizing that these are the \textit{real} roots of $char\left(H\left(k,l\right)\right)$. 

\vspace{.5cm}

\begin{center}
\begin{tabular}{ |c|c| } 
 \hline
   Root & Decimal Approx.\\
 \hline
  $\frac{1}{2}\left(h_1 + h_3 + \sqrt{h_1^2 + 4 h_2^2 - 2h_1h_3 + h_3^2}\right)$  & $z=1.6231$ at $k=1.1530, l=0.8641$ \\
 \hline
 $\frac{1}{2}\left(h_1 + h_3 - \sqrt{h_1^2 + 4 h_2^2 - 2h_1h_3 + h_3^2}\right)$ & $z=0.618034$ at $k=0.584444, l=0.797255$ \\
\hline
\end{tabular}

\label{table:roots}
\end{center} 

\vspace{.5cm}

The smaller of the two eigenvalues is $z= \frac{1}{2} \left(h_1 + h_3 - \sqrt{h_1^2 + 4 h_2^2 - 2h_1h_3 + h_3^2}\right)$. We call the real-valued minimizing pair $\tilde{z} =\left(0.584444, 0.797255\right)$. To find the minimizing integer pair, we identified candidate tuples by testing all possible pairs with entries given by the floor and ceiling of $\tilde{z}$: $\left(0,0 \right), \left( 0,1\right), \left( 1,0\right), \left( 1,1\right)$. Both $\left(0,0\right)$ and $ \left(1,1\right)$ minimize $z =\left(\frac{9-\sqrt{5}}{\left(2 \sqrt{2}\right)}\right)$ over the integers. Then, for $\left( k,l\right) \in \Z^2$ we have that the smallest eigenvalue is $z = \frac{9-\sqrt{5}}{2\sqrt{2}}$, roughly $0.6180$. 
Thus, we have a positive bound for the smallest growth in total distance of lattice points from $p$ under small perturbation. This eigenvalue is undefined at $\left(0.5, 0.5\right)$, but these are not integers and so this does not affect our computation. 
\qed

\subsubsection{Convex Functions}

\begin{proof}
	It remains to study the case
	\begin{equation}
		\sum_{\delta \in C_r} \phi \left(\| p - \delta\|\right)-  \sum_{\lambda \in A_r} \phi \left(\| p - \lambda\| \right).
	\end{equation}
	where $\phi$ is a convex function. 
	For $\lambda \in A_r$ and its corresponding point under perturbation, $\delta \in C_r$, consider the quantity
	\begin{equation}
		\| p - \delta\| = \|p - \lambda \| + \varepsilon_{\delta}, \qquad \varepsilon_{\delta} \in \R.
	\end{equation}
Rearranging, we have $ \| p - \delta\| - \|p - \lambda \| = r + \varepsilon_{\delta}$. Summing, we see
	\begin{equation}
	    \sum_{\delta \in C_r}\left(\| p - \delta\|\right)-  \sum_{\lambda \in A_r}\left(\| p - \lambda\|\right)= \sum_{\delta \in C_r}\left(\varepsilon_{\delta}\right).  
	\end{equation}
Then 
	\begin{equation}
	    \sum_{\delta \in C_r}\left(\|p-\delta\|\right)= \sum_{\lambda \in A_r}\left(\|p-\lambda\|\right)+ \sum \varepsilon_{\delta} 
	    = \sum_{\delta \in C_r}\left(r + \varepsilon_{\delta}\right)
	\end{equation}
A Taylor expansion of $\sum_{\delta \in C_r} \phi\left(r+ \varepsilon_{\delta}\right)$ around $\varepsilon_{\delta}= 0$ shows that
	\begin{align*}
        \sum_{\delta \in C_r} \phi \left(\| p - \delta\|\right)-  \sum_{\lambda \in A_r} \phi \left(\| p - \lambda\|\right)& = \sum_{\delta \in C_r}\left(r + \varepsilon_{\delta}\right)\\
		 & = \phi\left(r + \epsilon_{\delta}\right)+ \phi'\left(r\right)\sum_{\delta \in C_r} {\varepsilon_{\delta}} + \frac{\phi''\left(R \right)}{2} \sum_{\delta \in C_r}{\varepsilon_{\delta}^2} + o\left(d\left(\Lambda, \Delta \right)^2\right),
	\end{align*}
	where the error term is allowed to depend on $r$ and $A_r$. Our first result gives us a bound for $r \, |A_r| \, d\left(\Lambda, \Gamma \right)^2$, so we have
	\begin{equation}
		\sum_{\delta \in C_r} {\varepsilon_{\delta}}  \gtrsim r \, |A_r| \, d\left(\Lambda, \Delta \right)^2.
	\end{equation}
	Thus, we have shown that for a convex function $\phi$, 	\begin{equation}
		\sum_{\delta \in C_r} \phi \left(\| p - \delta\|\right)-  \sum_{\lambda \in A_r} \phi \left(\| p - \lambda\|\right)\gtrsim  r \, \phi'\left(r\right)\, |A_r| \, d\left(\Lambda, \Delta \right)^2.
	\end{equation}
\end{proof}

\newpage
   
\section{Further research}
\label{sect4}

\subsection{Higher dimensions}
This result gives promising indications of generalization. In particular, we conjecture that the $3$-dimensional analog of the result for some lattices which are optimal for sphere packing will hold. Our main obstacle in this endeavor is the growth of the dimension of the space of lattices. We construct $L\left(\R^n\right)$ as $SL\left(n, \R\right)/ SL\left(n, \Z\right)$, a quotient space with dimension $n^2-1$. The space $L\left(\R^2\right)$ considered above has dimension $3$, but if we consider lattices only up to rotation, then the quotient space is $2$ dimensional. The dimension of $L\left(\R^3\right)$ is $8$; even if we quotient by rotations again, the resulting space is $5$-dimensional.

\subsection{Connections to Sphere Packings}
Given a lattice $\Gamma$, we can associate a \textit{sphere packing} 
$\mathcal{B}$ by putting spheres of the same radius around each lattice point so that the resulting spheres are mutually tangent. Informally, optimal sphere packings in Euclidean spaces are arrangements of (disjoint)
spheres of the same size which cover as much of the space as possible. More precisely, let $B_r\left(x \right)$ denote a Euclidean ball of radius $r$ around the point $x$. 
We define $B_r\left(x, \Lambda \right): = B_r\left(x\right) \cap \mathcal{B}$. Then the ratio of the volumes $$ \frac{B_r\left(x, \Lambda \right)}{B_r\left(x \right)}$$ is called the \textit{density} of the packing. An \textit{optimal} packing maximizes $$r_{\mathcal{B}} = \lim_{r \rightarrow \infty} \frac{B_r\left(x, \Lambda \right)}{B_r\left(x \right)}.$$
The sphere packings associated to the square and hexagonal lattices in $\R^2$ are  critical points for this notion of density, and as we showed above, the lattices critical points in the space of lattices for our problem of studying distances to deep holes. It is natural to ask whether those lattices in $L \left(\R^n \right)$ which are associated to optimal sphere packings are also extremal in our sense. Generally, we conjecture that any unimodular lattice which is also an optimal sphere packing in $\R^n$ will exhibit this extremal property.

\subsection{Other point sets} There are other naturally occurring families of point sets in Euclidean spaces, arising from various geometric and dynamical constructions. Examples include sets of \emph{holonomy vectors of saddle connections on translation surfaces} \cite{holonomy}, and \emph{cut-and-project quasicrystals} \cite{quasi}. In both examples there are versions of the question we have considered above about \emph{deep holes} in these point sets; however, as we saw with the growth on dimension of $L \left(\R^n \right)$, understanding optimal configurations is a challenging question due to the higher-dimensional nature of the associated spaces of configurations. An additional consideration would be the lack of an obvious additive structure. Intuition may be gained from first examining examples like the sets of saddle connections associated to Veech surfaces \cite{holonomy} and well-known tilings \cite{tilings}, like the Penrose tiling\cite{penrose}.

\newpage

\printbibliography  

\end{document}